\documentclass[a4paper]{amsart}
\usepackage{amsmath,amsthm,amssymb, amscd, amsfonts}

\theoremstyle{plain}
\newtheorem{theorem}{Theorem}[section]
\newtheorem{corollary}[theorem]{Corollary}

\newtheorem{proposition}[theorem]{Proposition}
\newtheorem{lemma}[theorem]{Lemma}

\theoremstyle{definition}
\newtheorem{definition}[theorem]{Definition}

\newtheorem{example}[theorem]{Example}

\theoremstyle{remark}
\newtheorem{remark}[theorem]{Remark}

\numberwithin{equation}{section}\theoremstyle{plain}

\newcommand{\C}{{\mathcal C}}
\newcommand{\B}{{\mathcal B}}
\newcommand{\D}{{\mathcal D}}

\newcommand{\Z}{{\mathcal Z}}

\newcommand\SL{\operatorname{SL}}
\newcommand\rep{\operatorname{rep}}
\newcommand\cd{\operatorname{cd}}
\newcommand\Irr{\operatorname{Irr}}
\newcommand\FPdim{\operatorname{FPdim}}
\newcommand\vect{\operatorname{Vec}}
\newcommand\id{\operatorname{id}}
\newcommand\ad{\operatorname{ad}}

\newcommand\ord{\operatorname{ord}}

\newcommand\End{\operatorname{End}}

\newcommand\rk{\operatorname{rk}}

\newcommand\Hom{\operatorname{Hom}}
\newcommand\Rker{\operatorname{Rker}}
\newcommand\RZ{\operatorname{RZ}}
\newcommand\Aut{\operatorname{Aut}}

\newcommand{\1}{\textbf{1}}

\begin{document}
\title[Faithful simple objects, orders and gradings of fusion categories]{Faithful simple objects,
orders and gradings of fusion categories}
\author{Sonia Natale}
\address{Sonia Natale: Facultad de Matem\'atica, Astronom\'\i a y F\'\i sica.
Universidad Nacional de C\'ordoba. CIEM -- CONICET. Ciudad
Universitaria. (5000) C\'ordoba, Argentina}
\email{natale@famaf.unc.edu.ar
\newline \indent \emph{URL:}\/ http://www.famaf.unc.edu.ar/$\sim$natale}

\thanks{Partially supported by CONICET and SeCYT--UNC}


\date{October 7, 2011.}

\begin{abstract} We establish some relations between the
orders of simple objects in a fusion category and the structure of
its universal grading group. We consider fusion categories which
have a faithful simple object and show that its universal  grading
group must be cyclic. As for the converse, we prove that a braided
nilpotent fusion category with cyclic universal grading group
always has a faithful simple object. We study the universal
grading of fusion categories with generalized Tambara-Yamagami
fusion rules. As an application, we classify modular categories in
this class and describe the
modularizations of braided Tambara-Yamagami fusion categories.
\end{abstract}

\maketitle

\section{Introduction}

Group gradings on fusion categories and more precisely, group
extensions of a fusion category, are key ingredients in several
classification results. In particular, they underlie the notions
of nilpotency and solvability of a fusion category, developed in
\cite{gel-nik, ENO2}. Group extensions of fusion categories have
been recently classified in \cite{ENO-homotopy}.

\medbreak We shall work over an  algebraically closed base field $k$ of
characteristic zero.  Let $\C$ be a fusion category over $k$.
There is a canonical faithful grading $\C = \oplus_{g \in
U(\C)}\C_g$, called the \emph{universal grading} of $\C$, with trivial homogeneous component $\C_e = \C_{\ad}$, where $\C_{\ad}$ is the \emph{adjoint subcategory} of $\C$, that is, the fusion
subcategory generated by $X \otimes X^*$, where $X$ runs
through the simple objects of $\C$ \cite{ENO}. The group $U(\C)$ is called the
\emph{universal grading group} of $\C$ \cite{gel-nik}.

\medbreak Let $X$ be an object of $\C$. Then $X$ is called
\emph{faithful} if the fusion subcategory $\C[X]$ generated by $X$
is all of $\C$. So that $X$ is faithful if and only if every
simple object of $\C$ appears with positive multiplicity in some
tensor power of $X$.

\medbreak Let $G$ be a finite group. An example of a fusion
category over $k$ is given by the category $\rep G$  of finite
dimensional $k$-linear representations of $G$. In this case, the
universal grading group is isomorphic to the center $Z(G)$ of $G$
and the adjoint subcategory coincides with the category $\rep
G/Z(G)$. See \cite{gel-nik}.

Suppose $V$ is a finite dimensional representation of $G$, and let
$\chi$ be the character of $V$. Then it follows from the
Burnside-Brauer Theorem (see \cite[Theorem 4.3]{isaacs}) that $V$
is a faithful object of $\rep G$ if and only if $\chi$ is a
faithful character of $G$, that is, if and only if $\ker \chi =
1$.

\medbreak Let $X$ be a nonzero object of a fusion category $\C$.
The \emph{order} of $X$ is the smallest natural number $n$ such
that the $n$th tensor power $X^{\otimes n}$ contains the trivial
object $\1$ of $\C$. The order of every nonzero object in $\C$ is
finite; indeed, it is not bigger than the rank of $\C$. This
invariant is introduced and studied in \cite{KSZ} for the category
of representations of a semisimple Hopf algebra.

The orders of simple objects play a r\^ ole in recent
classification results \cite{RSW, hong-rowell}. As pointed out in
\cite{hong-rowell}, the classification of modular categories of a
given rank divides naturally into those for which every simple
object is self-dual, that is, of order $\leq 2$, and those for
which at least one simple object is not self-dual, that is, of
order bigger than $2$. See, in particular, Theorem 2.2 of
\textit{loc. cit.}

\medbreak In this paper we establish some relations between the
structure of the group $U(\C)$ and the orders of the simple
objects of $\C$. We show that if $\C$ is generated by simple
objects $X_1, \dots, X_m$, then the group $U(\C)$ is generated by
elements $g_1, \dots, g_m$, such that $X_i \in \C_{g_i}$, and the
order of $g_i$ divides the order of $X_i$, for all $i = 1, \dots,
m$. Hence, if $\C$ has a faithful simple object $X$, then the
universal grading group of $\C$ is cyclic and its order divides
the order of $X$.

We also establish the converse in the case where $\C$ is braided
and nilpotent. That is, if such a fusion category has a cyclic
universal grading group, then it has a faithful simple object.
These results are contained in Section \ref{s-faithful}; see
Theorems \ref{G-cyclic} and \ref{braided-primepower}. They extend
classical results for finite groups. The proof of Theorem
\ref{braided-primepower} relies on the one hand on Theorem
\ref{braided-nilp} which generalizes, in the braided case, the
fact that a nontrivial normal subgroup of a finite nilpotent group
intersects the center nontrivially, and on the other hand, on the
fact that every braided nilpotent fusion category is equivalent to
a tensor product of braided fusion categories whose
Frobenius-Perron dimensions are powers of distinct primes
\cite[Theorem 1.1]{DGNO}.

\medbreak Let us summarize some consequences of these facts. See
Corollaries \ref{sd-faithful}, \ref{exp-udec}, \ref{gdec} and
Proposition \ref{gen-sd}. Firstly, for any fusion category $\C$,
the order of $U(\C)$ divides the least common multiple of the
orders of simple objects of $\C$. Thus, if every simple object is
self-dual, then $U(\C)$ is an elementary abelian $2$-group.

Also, if $\C$ is nilpotent and has
a simple object of order $p$, where $p$ is a
prime number, then the Frobenius-Perron dimension of
$\C$, which is always an integer, is divisible by $p$. Moreover, if the order of every simple object of $\C$ is
a power of $p$, then the Frobenius-Perron dimension of $\C$ is a
power of $p$.

Assume in addition that $\C$ is a braided. If $X_1, \dots, X_n$
are simple objects that generate $\C$ as fusion category,  we get
that the exponent of $U(\C)$ divides the least common multiple of
the orders of $X_1, \dots, X_n$. In particular, $U(\C)$ is an
elementary abelian $2$-group if $\C$ is generated by self-dual
simple objects.

These results can be formulated in terms of the group of
invertible objects of $\C$ when $\C$ is a modular category, since
in this case there is an isomorphism between this group and
$U(\C)$ \cite{gel-nik}.

\medbreak In the last two sections of the paper we discuss some
applications of the relations between the orders of simple objects
and the orders of the elements of the universal grading group to
study a class of fusion categories. We consider fusion categories
$\C$ with \emph{generalized Tambara-Yamagami fusion rules} in the
sense that $\C$ is not pointed and the tensor product of two
simple objects of $\C$ is a sum of invertible objects. Fusion
categories with these kind of fusion rules are classified, up to
equivalence of tensor categories, in \cite{liptrap}.

In Section \ref{generalized-ty} we discuss the universal grading
of this kind of fusion categories.  We then apply the results of
this section to establish a classification result for  modular
categories in this class. Specifically, we show that if $\C$ is a
modular category, then $\C$ has generalized Tambara-Yamagami fusion
rules if and only if $\C$ is equivalent to $\mathcal I \boxtimes \mathcal B$,
where $\mathcal I$ is an Ising category and $\mathcal B$ is a
pointed modular category. See Theorem \ref{modular-genty}. This
implies the classification of such modular categories in terms of
group theoretical data, in view of the results on Ising and
pointed modular categories \cite[Subsection 2.11 and Appendix
B]{DGNOI}, \cite{quinn}.

\medbreak Let $\C$ be a Tambara-Yamagami fusion category. Up to
isomorphism, $\C$ has exactly one non-invertible object $X$, and a
(necessarily abelian) group of invertible objects $G$ such that $X
\otimes X \simeq \oplus_{g \in G}g$. The classification of these
categories, up to tensor equivalence, is given in \cite{TY}. The
possible structures of a braided and spherical fusion category in
$\C$ are classified in \cite{Siehler-braided}. By \cite[Theorem
1.2]{Siehler-braided}, $\C$  admits a braiding if and only if $G$
is an elementary abelian $2$-group. In particular, $\FPdim \C =
2^{m+1}$, where $m$ is the rank of $G$. Moreover, by \cite[Theorem
1.2 (3)]{Siehler-braided} each braiding of $\C$ has two choices of
ribbon structures compatible with it.

Let $\C$ be a braided Tambara-Yamagami fusion category and let us
regard $\C$ as a premodular category with respect to a fixed
ribbon structure $\theta \in \Aut(\id_\C)$. See \cite{bruguieres},
\cite{mueger}.
We show that $\C$ is modularizable and describe its modularization
$\tilde \C$. We prove that $\tilde \C$ is pointed if and only if
$\C$ is integral. Otherwise, $\tilde \C \simeq \mathcal I$, where
$\mathcal I$ is an Ising category. This is contained in Section
\ref{ty-mod}. Theorem \ref{modulariz-ty} is the main result of
this section. Its proof relies on the fact that, unless it is
symmetric, $\C$ is a $T$-equivariantization of $\tilde \C$ for a
certain subgroup $T$ of index $\leq 2$ of the group $G$ of
invertible objects of $\C$.

\medbreak Recall from \cite{Siehler} that if $G$ is a finite group
and $\kappa$ is a nonnegative integer, a \emph{near-group fusion
category} of type $(G, \kappa)$ is a fusion category $\C$ whose
isomorphism classes of simple objects are represented by $G$ and a
non-invertible object $X$, obeying
\begin{equation}\label{near-group}g \otimes h \simeq gh, \quad
X \otimes X \simeq \oplus_{g \in \Gamma} g \oplus \kappa X, \quad
\forall g, h \in \Gamma.\end{equation} Near-group categories of
type $(G, 0)$ are thus Tambara-Yamagami categories.

For the type $(G, \kappa)$ with $\kappa > 0$, all possible
structures of braided fusion category in $\C$ are classified in
\cite{thornton}. In this case, if $\C$ is not symmetric, then
either $G = 1$ or $G$ is isomorphic to $\mathbb Z_2$ or $\mathbb
Z_3$. It follows also from the results of \cite{thornton} that any
braided near group category is a premodular category and, as such,
it is modularizable. Furthermore, the modularization of $\C$ is a
pointed fusion category, unless $\C$ is of type $(1, 1)$, also
called a \emph{Yang-Lee} category, in which case $\C$ is modular.

It turns out that the modularization of any braided near-group
category is always either pointed or self-dual of rank at most
$3$.

\subsection*{Conventions and notation} Let $\C$ be a fusion category over
$k$. That is, $\C$ is a $k$-linear semisimple rigid tensor category $\C$
with finitely many isomorphism classes of simple objects,
finite-dimensional spaces of morphisms, and such that the unit
object $\textbf{1}$ of $\C$ is simple.
The set of isomorphism classes of simple objects in $\C$
will be denoted by $\Irr (\C)$. By abuse of notation, we shall
indicate a simple object and its isomorphism class by the same letter.

If $X \in \Irr \C$ and $Y$ is an object of $\C$, the multiplicity
 of $X$ in $Y$ will be denoted by $m(X, Y)$. Thus $m(X, Y) = \dim \Hom_\C(X, Y)$ and
 $Y \simeq \oplus_{X \in \Irr(\C)} m(X, Y) X$.

The group of isomorphism classes of invertible objects of $\C$
will be denoted by $G(\C)$.
The group $G(\C)$ acts on the set $\Irr (\C)$ by left
(and right) tensor multiplication. Denote by $G[X]$ the stabilizer
of  $X \in \Irr (\C)$ under the left action of $G(\C)$.
If $g \in G(\C)$ and $X \in \Irr (\C)$, we have $g \in G[X]$ if and
only if $m(g, X \otimes X^*) > 0$ if and only if $m(g, X \otimes
X^*) = 1$.

A \emph{fusion subcategory} of $\C$ is a full tensor subcategory
stable under direct sums and subobjects. A fusion subcategory is
itself a fusion category \cite[Corollary F.7]{DGNOI}. Let $S$ be a
set of objects of $\C$. The smallest abelian subcategory of $\C$
containing $S$ and stable by direct sums and subobjects will be
denoted $(S)$. We shall indicate by $\C[S]$ the fusion subcategory
of $\C$ generated by $S$. The fusion subcategory generated by $\{
X\}$, where $X$ is an object of $\C$ will be denoted $\C[X]$. The
maximal pointed fusion subcategory of $\C$ is denoted $\C_{pt} =
\C[G(\C)]$.

\section{Group gradings on fusion categories}

Let $\C$ be a fusion category and let $G$ be a finite group. A \emph{$G$-grading} on $\C$
is a decomposition of $\C$ as a direct sum of full abelian
subcategories $\C = \oplus_{g \in G} \C_g$, such that $\C_g^* =
\C_{g^{-1}}$ and  the tensor product $\otimes: \C \times \C \to
\C$ maps $\C_g \times \C_h$ to $\C_{gh}$.  The neutral component
$\C_e$ is thus a fusion subcategory of $\C$.

A $G$-grading on $\C$ is equivalently determined by a function $\lambda: \Irr
(\C) \to G$ such that $\lambda(X^*) = \lambda(X)^{-1}$ and
$\lambda(Z) = \lambda(X)\lambda(Y)$, for all $X, Y, Z \in \Irr
(\C)$ such that $m(Z, X \otimes Y) > 0$.

\medbreak The grading $\C = \oplus_{g \in G} \C_g$ is called \emph{faithful} if $\C_g \neq 0$, for all $g
\in G$. In other words, the associated map $\lambda: \Irr(\C) \to G$ is surjective. In this case, $\C$ is called a \emph{$G$-extension} of
$\C_e$ \cite{ENO2}.
When $\C$ is a $G$-extension of a fusion subcategory $\D = \C_e$,
the Frobenius-Perron dimensions of $\C_g$ are all equal and we
have $\FPdim \C = |G| \FPdim \D$ \cite[Proposition 8.20]{ENO}.

\medbreak Let us recall the notion of nilpotent fusion category
from \cite{gel-nik}. The \emph{adjoint subcategory} $\C_{\ad}$ is
the full tensor subcategory of $\C$ generated by $X\otimes X^*$,
$X \in \Irr \C$. The \emph{upper central series} $\C = \C^{(0)}
\supseteq \C^{(1)}\supseteq \ldots \supseteq \C^{(n)} \supseteq
\ldots$ of $\C$ is defined inductively as $\C^{(n)} =
(\C^{(n-1)})_{\ad}$, for all $n \geq 1$. A fusion category $\C$ is
called \emph{nilpotent} if its upper central series converges to
$\vect_k$, that is, if $\C^{(n)} = \vect_k$ for some $n \geq 0$.
The smallest such $n$ is called the \emph{nilpotency class} of
$\C$.

For instance, if $p$ is a prime number and $\C$ is any fusion
category with $\FPdim \C = p^n$, $n \geq 0$, then $\C$ is
nilpotent \cite{gel-nik, ENO}.

\medbreak  There is a faithful grading $\C = \oplus_{g \in
U(\C)}\C_g$, where $\C_e = \C_{\ad}$ and $U(\C)$ is the universal
grading group of $\C$.  Let $\lambda_\C: \Irr(\C) \to U(\C)$
denote the universal grading of $\C$. Its universal property can
be stated as follows: for any grading $\lambda: \Irr(\C) \to G$ by
a group $G$, there exists a unique group homomorphism $\phi: U(\C)
\to G$ such that $\lambda = \phi\lambda_\C$.

\medbreak Suppose $\D$ is a fusion subcategory of $\C$.  Then $\D$ is faithfully graded by the subgroup
$U_\D(\C) = \{ g \in U(\C)|\,  \D\cap \C_g \neq 0 \} \subseteq
U(\C)$. By the universal property of $U(\D)$ there is a
surjective group homomorphism $\phi_\D:U(\D) \to U_\D(\C)$
\cite[Corollary 3.7]{gel-nik}. For all $g \in U_\D(\C)$, we have a
decomposition $\D \cap \C_g = \oplus_{\phi_\D(t) = g}\D_t$.

\medbreak Let $G$ be a finite group. The category $\C = \rep G$
of finite dimensional representations of $G$ is nilpotent if and
only if $G$ is nilpotent \cite{gel-nik}. On the other hand, a
fusion subcategory of $\C$ is of the form $\rep G/N$ for some
normal subgroup $N$ of $G$. The next theorem amounts in this case
to the well-known fact that  if $G$ is nilpotent and $N \neq 1$ is
a subgroup, then $N \cap Z(G) \neq 1$.

\begin{theorem}\label{braided-nilp} Let $\C$ be a nilpotent fusion category with commutative Grothendieck ring.
Suppose $\D$ is a fusion subcategory of $\C$ such that $U_\D(\C) =
U(\C)$, that is, $\D \cap \C_g \neq 0$, for all $g \in U(\C)$.
Then $\D = \C$.
\end{theorem}

Note that the theorem applies, in particular, if $\C$ is braided.

\begin{proof} Note first that if $\C^{(1)} = \C_{\ad} \subseteq \D$, then $\D = \C$.
Indeed, in this case $\D$ is a $\C_{\ad}$-sub-bimodule category of
$\C$ and therefore,  it is a sum of indecomposable sub-bimodule
categories $\D = \oplus_{g \in S}\C_g$, for some subset $S
\subseteq U(\C)$ (since the subcategories $\C_g$ are the
indecomposable $\C_{\ad}$-sub-bimodule categories of $\C$
\cite{gel-nik}). Hence $\D \cap \C_g = 0$, for all $g \notin S$,
implying, by assumption, that $S = U(\C)$. That is, $\D =
\oplus_{g \in U(\C)}\C_g = \C$.

\medbreak Consider the upper central series
$\dots \subseteq \C^{(m+1)} \subseteq \C^{(m)} \subseteq \dots \subseteq \C^{(1)} \subseteq \C^{(0)} = \C$.
Since $\C$ is nilpotent, there exists $n \geq 1$ such that
$\C^{(n)} = \vect_k$. In particular, $\C^{(n)} \subseteq \D$. Let
$m$ be the smallest nonnegative integer such that $\C^{(m)}
\subseteq \D$. We may assume that $m \geq 1$. Then $\C^{(m-1)}
\nsubseteq \D$ and $\C^{(m)} = (\C^{(m-1)})_{\ad} \subseteq \D$.

\medbreak Let $\D^{co}$ be \emph{commutator} subcategory of $\D$ in $\C$, that is, $\D^{co}$ is the fusion subcategory of $\C$ generated by all simple objects $Y$ of $\C$ such that $Y \otimes Y^*$ belongs to $\D$ \cite[Definition 4.10]{gel-nik}. Thus  $\mathcal (\D^{co})_{\ad} \subseteq \D \subseteq \D^{co}$, and $\D^{co}$ is the largest fusion subcategory of $\C$ with the property $\mathcal (\D^{co})_{\ad} \subseteq \D$.

For every $j = 0, \dots$, let $\D^{[j]}$ be the fusion subcategory defined as follows: $\D^{[0]} = \D$, $\D^{[j+1]} = (\D^{[j]})^{co}$, $j\geq 0$. As before, we have
$(\D^{[j+1]})_{\ad} \subseteq \D^{[j]} \subseteq \D^{[j+1]}$, for all $j = 0, 1, \dots$.

Since $\D \subseteq \D^{[j]}$,  then we get
\begin{equation}\label{fgg}\D^{[j]} \cap \C_g \neq 0, \quad \forall g \in U(\C), \; \forall j = 0, \dots \end{equation}
On the other hand, we have inclusions
\begin{equation}\label{inclusion} \C^{(m-j)} \subseteq \D^{[j]}, \quad \forall j = 0, \dots, m. \end{equation}
Indeed,  $\C^{(m)} = (\C^{(m-1)})_{\ad} \subseteq \D$, so that
$\C^{(m-1)} \subseteq \D^{co} = \D^{[1]}$. Assuming inductively
that $\C^{(m-j)} \subseteq \D^{[j]}$ for $1 \leq j < m$, we get
that $(\C^{(m-(j+1))})_{\ad} = \C^{(m-j)} \subseteq \D^{[j]}$.
Hence $\C^{(m-(j+1))} \subseteq (\D^{[j]})^{co} = \D^{[j+1]}$.

It follows from \eqref{inclusion} that $\C^{(1)} = \C_{\ad}
\subseteq \D^{[m-1]}$.  Therefore
$\D^{[m-1]} = \C$, in view of \eqref{fgg}. We now show by
induction that $\D^{[m-j]} = \C$, for all $j = 1, \dots, m$.
Suppose that $\D^{[m-j]} = \C$. Then $\C_{\ad} =
(\D^{[m-j]})_{\ad} \subseteq \D^{[m-(j+1)]}$. Hence
$\D^{[m-(j+1)]} = \C$.
In particular, $\D = \D^{[0]} = \C$. This finishes the proof of
the theorem.
\end{proof}

\section{The fusion subcategory generated by  a simple object}

Let $\C$ be a fusion category. Recall that an object $X$ of $\C$
is called \emph{faithful} if  for every $Y \in \Irr (\C)$, we have
$m(Y, X^{\otimes n}) > 0$, for some integer $n \geq 1$. Thus $X
\in \C$ is faithful if and only if  $\{ X\}$ generates $\C$ as a
fusion category.


\medbreak The following definition appears in \cite[Chapter
4]{KSZ}.

\begin{definition} Let $X$ be a nonzero object of $\C$.
The smallest nonnegative integer $n$ such that $m(\1, X^{\otimes n}) >0$
is called the \emph{order} of $X$. We shall denote it by
$\ord(X)$. \end{definition}

Suppose $X \in \Irr (\C)$. Then $\ord(X) = 1$ if and only if $X =
\1$. Also, $\ord(X) = 2$ if and only if $X = X^*$ and $\1 \neq X$.

\begin{remark} Recall that the \emph{rank} of $\C$, denoted $\rk (\C)$, is the cardinality of the set  $\Irr (\C)$.
It is shown in \cite[Proposition 5.1]{KSZ} that an object $X$ is
faithful if and only if the matrix of left multiplication by $X$
in the basis $\Irr (\C)$ of the Grothendieck ring of $\C$ is
\emph{indecomposable}. As in \cite[Corollary 5.1]{KSZ}, we get
that $\ord(X) \leq \rk (\C)$, for all nonzero object $X$. Indeed,
these results are stated \textit{loc. cit.} for the case where
$\C$ is the category of finite dimensional representations of a
semisimple Hopf algebra, but the proof applies \textit{mutatis
mutandis} in any fusion category as well.

\medbreak Suppose that $\C$ is a spherical fusion category. Then,
by \cite[Corollary 5.13]{ng-scha},  for every simple object $X$ of
$\C$, we have furthermore that $\ord(X) \leq \textrm{FSexp }(\C)$,
where $\textrm{FSexp }(\C) \in \mathbb N$ is the
\emph{Frobenius-Schur exponent} of $\C$.
\end{remark}

\medbreak Let $X$ be an object of $\C$ and let $\C[X]$ be the
fusion subcategory of $\C$ generated by $X$. The universal grading
group of $\C[X]$ will be denoted by $U(X)$.

\begin{remark} Let $\C$ be a fusion category and assume $\C$ has a unique maximal fusion subcategory.
Then $\C$ has a faithful simple object.

\begin{proof} Let $\D \subsetneq \C$ be the unique maximal fusion subcategory of $\C$.
If $X \in \Irr (\C)$ is such that $\C[X]\subsetneq \C$, then
$\C[X] \subseteq \D$.  Hence there must exist some $X \in \Irr
(\C)$ with $\C[X] = \C$. This proves the claim.
\end{proof} \end{remark}

\begin{example} (i) Let $\C$ be a pointed fusion category with a finite group $G$ of invertible objects. Then a simple object $g \in G$ is faithful if and only if $G$ is a cyclic group generated by $g$.

\medbreak (ii) Suppose $\C = \rep G$, where $G$ a finite group.
Let $X$ be an object of $\C$, and let  $\chi = \chi_X$ be the
character of $X$. Then $X$ is a faithful object of $\C$ if and
only if $\chi$ is a faithful character of $G$.

We have in this example $\C[X] = \rep (G/\ker \chi)$. Moreover, if
$X \in \Irr (\C)$ (that is, $X$ is an irreducible representation
of $G$), then $U(X) \simeq Z(G/\ker \chi) \simeq Z(\chi)/\ker
\chi$ \cite[Lemma 2.27]{isaacs}. Recall that the normal subgroups
$\ker \chi$ and $Z(\chi)$ are defined as \begin{equation}\ker \chi
= \{ g \in G|\, \chi(g) = \chi(1) \}, \quad Z(\chi) = \{ g \in
G|\, |\chi(g)| = (\chi\chi^*)(g) = \chi(1) \}.\end{equation}

\medbreak (iii) Let $H$ be a semisimple Hopf algebra and let $\C =
\rep H$ be the category of finite dimensional representations of
$H$. Let $\pi: H \to \End(V)$ be a finite dimensional
representation. The subspaces
\begin{align*}& \Rker(\pi)  = \{ h \in H|\, \pi(h_{(1)}) \otimes h_{(2)} = \id_V \otimes h \}, \\
& \RZ(\pi)  = \Rker(\pi  \otimes  \pi^*) = \{ h \in H|\, (\pi
\otimes \pi^*)(h_{(1)}) \otimes h_{(2)}  = \id_{V\otimes V^*}
\otimes h  \},\end{align*} are normal right coideal subalgebras of
$H$. (Similarly one can define normal left coideal subalgebras
$\operatorname{Lker}(\pi)$ and $\operatorname{LZ}(\pi)$ with
analogous properties.) The coideal subalgebras $\Rker(\pi)$ and
$\operatorname{Lker}(\pi)$ are studied in \cite{burciu}.

\medbreak The quotient Hopf algebra $H/H(\Rker \pi)^+$  satisfies
$\C[\pi] = \rep (H/H(\Rker \pi)^+)$. We have in addition:

\begin{proposition}\label{a-central} Let $A \subseteq H$ be a right coideal subalgebra.
Then $A \subseteq \RZ(\pi)$ if and only if, for all $a \in A$,
$\pi (a) = \lambda(a) \id_V$,  for some linear character $\lambda
\in G(A^*)$.  In particular, if $A \subseteq \RZ(\pi)$ and $\pi$
is faithful then $A \subseteq Z(H)$. \end{proposition}

\begin{proof} Note that if an element $h \in H$ belongs to $\RZ(\pi) = \Rker (\pi \otimes \pi^*)$,
then $h$ acts trivially on $V \otimes V^*$. Since $V \otimes V^*
\simeq \End(V)$ as $H$-modules,  this implies that there exists
$\lambda(h) \in k$ such that $\pi(h) = \lambda(h) \id_V$.  In
other words, the normal right coideal subalgebra $\RZ(\pi)$ acts
by scalars on $V$. Clearly, the map $\lambda$ defines a linear
character on $\RZ(\pi)$. This implies the only if direction.

Conversely, suppose that for all $a \in A$, $\pi (a) = \lambda(a)
\id_V$,  for some linear character $\lambda \in G(A^*)$. Then $A$
acts trivially on $\End(V) \simeq V \otimes V^*$ and since $A$ is
a right coideal subalgebra, we get that $A \subseteq \RZ(\pi)$.
This finishes the proof of the lemma. \end{proof}

Let $K \subseteq H$ be the maximal central Hopf subalgebra. It
follows from \cite[Theorem 3.8]{gel-nik} that $K \simeq k^{U(H)}$,
where $U(H)$ is the universal grading group of $\C = \rep H$. Thus
$U(H)$ is isomorphic to the group $G(K^*)$ of one dimensional
representations of $K$.

\begin{corollary} Let $H$ be a semisimple Hopf algebra and let $U(H)$
be the universal grading group of the category $\rep H$.  Then
$k^{U(H)}$ is the unique maximal Hopf subalgebra of $H$ contained
in $\cap_{\pi \in \Irr(H)}\RZ(\pi)$.
\end{corollary}

\begin{proof} As remarked above, we have an isomorphism $k^{U(H)} \simeq
K$. Since $K$ is a central Hopf subalgebra, we have   $K \subseteq
\cap_{\pi \in \Irr(H)}\RZ(\pi)$, by Proposition \ref{a-central}.
Let $A \subseteq H$ be a Hopf subalgebra contained in $\cap_{\pi
\in \Irr(H)}\RZ(\pi)$. Then $A$ acts as scalar multiplication by a
linear character on any irreducible representation. This implies
that $A \subseteq Z(H)$, since irreducible representations
separate points of $H$. Hence $A \subseteq K$. This proves the
corollary. \end{proof}
\end{example}

Let $\C$ be a fusion category and let $X\in \Irr (\C)$. Recall
that by the universal property of $U(X)$ there is a group
homomorphism $\phi_X = \phi_{\C[X]}: U(X) \to U(\C)$. This is
determined as follows: for every $t \in U(X)$, $\phi_X(t) = g \in
U(\C)$ if and only if $\C[X]_t \subseteq \C_g$. The following
lemma will be used later on.

\begin{lemma}\label{udeC} We have $U(\C) = \cup_{X \in \Irr (\C)}\phi_X(U(X))$.  \end{lemma}

\begin{proof} Let $g \in U(\C)$ and let $Y$ be a simple object in $\C_g$.
By definition of $\phi_Y$, we have $g \in \phi_Y(U(Y))$. This
proves the lemma. \end{proof}

We end this section by recalling some known families of examples.
Let $k = \mathbb C$.

\begin{example}{(Verlinde categories for $\mathfrak{sl}_2$.)}
Let $n$ be a nonnegative integer and let $q = e^{\frac{i
\pi}{n+2}}$.  Let $\C_n = \C_n(\mathfrak{sl}_2)$ be the
semisimplification of the category of representations of
$U_q(\mathfrak{sl}_2)$ \cite{AP, BK}. It is well-known that $\C_n$
is a modular fusion category over $k$. Isomorphism classes of
simple objects in $\C_n$ are represented by objects $X_i$, $0 \leq
i \leq n$, with $X_0 = \1$, $X_i^* = X_i$, and obeying the
truncated Clebsch-Gordan fusion rules:
\begin{equation}\label{FR} X_i \otimes X_j \simeq \oplus_{l = |i-j|, i+j \equiv l (2)}^{min(i+j, 2n-(i+j))} X_l. \end{equation}
The Frobenius-Perron dimension of $X_j$ is given by $\FPdim X_j =
\frac{q^{j+1}-q^{-(j+1)}}{q-q^{-1}} = \frac{\sin
((j+1)\theta)}{\sin (\theta)}$, where $\theta = \frac{\pi}{n+2}$.
In particular, there are exactly two invertible objects $\1 = X_0$
and $g = X_n$.

There is a faithful $\mathbb Z_2$-grading on $\C_n$ given by $\C_n
= \C_n^+ \oplus \C_n^-$,  where $\C_n^{\pm}$ is the full abelian
subcategory with simple objects $X_i$, $i$ even (respectively,
odd). Letting $X = X_1$, relation \eqref{FR} implies that $\C_n =
\C_n[X]$, so that $X$ is a faithful simple object of order $2$. We
have in this example $U(\C_n) \simeq G(\C_n) \simeq \mathbb Z_2$.
\end{example}

\begin{example}{(Modular near-group categories.)}
Let $\C$ be a near-group category of  type $(G, \kappa)$, as
described in the Introduction. Then  we have $U(\C) = 1$ if
$\kappa > 0$, and $U(\C) = \mathbb Z_2$ if $\kappa = 0$. Indeed,
in the first case, it is clear from \eqref{near-group} that
$\C_{\ad} = \C$, while in the second case $\C_{\ad} \simeq \C[G]$.

\medbreak Let $\C$ be a near-group category of type $(G, \kappa)$
and suppose that $\C$ admits a \emph{modular} structure.  Then  $G
= G(\C) \simeq U(\C)$ is of order $1$ or $2$.

\medbreak If $\kappa \neq 0$, then $G = 1$, so $\C$ is of rank $2$
and has a non-invertible object $X$, such that $X^{\otimes 2} = \1
\oplus X$. By \cite{ostrik-rk2} there are $4$ nonequivalent
braided fusion categories with these fusion rules, called
\emph{Yang-Lee} categories, and they are modular.

If $\kappa = 0$, then $G \simeq \mathbb Z_2$. The fusion rule
\eqref{near-group} is in  this case $X^{\otimes 2} = \1 \oplus a$,
where $\langle a \rangle = G$. Fusion categories with these fusion
rules are called \emph{Ising categories}. They are classified  in
\cite[Appendix B]{DGNOI}. In particular, every braided Ising
category is modular \cite[Corollary B.12]{DGNOI}. \end{example}

\begin{example} (Fermionic Moore-Read fusion rule.) Let $G = \langle g|\, g^4 = e \rangle$
denote the cyclic group of order $4$. Consider the commutative
fusion rules on the set $G \cup \{ X, X'\}$ determined by $h
\otimes h' \simeq hh'$, $h, h' \in G$, and
\begin{align*}& g^2 \otimes X \simeq X, \quad g^2 \otimes X' \simeq
X', \quad X \otimes X' \simeq \1 \oplus g^2, \\
& g \otimes X \simeq X', \quad g \otimes X' \simeq
X, \quad X \otimes X \simeq g \oplus g^3, \\
& g^3 \otimes X \simeq X', \quad g^3 \otimes X' \simeq X, \quad X'
\otimes X' \simeq g \oplus g^3. \end{align*} It is known that, up
to equivalence of tensor categories, there are four fusion
categories with these fusion rules, non of them braided
\cite{bonderson, liptrap}. If $\C$ is any fusion category
satisfying these fusion rules, then $\FPdim \C = 8$ and $\C$ is
nilpotent of nilpotency class $2$. We have in addition $G = G(\C)
\simeq U(\C)$ and $\C_{\ad} = \C[g^2]$.

Observe that in these examples $g^2$ is the only simple object of
order $2$. Also, $X$ is of order $4$ and $X^{\otimes 2}$
decomposes as a direct sum of simple objects of order $4$.
\end{example}

\begin{example}{(Faithful simple objects of order $2$.)}
Let $\C$ be a fusion category over $k$ and assume $X$ is a
faithful simple object of $\C$ of order $2$,  that is, $\C =
\C[X]$, $X  \neq \1$ and $X \simeq X^*$.

\medbreak  Let $q \in k^{\times}$ and consider the tensor category
$\rep \SL_q(2)$ of finite dimensional comodules over the Hopf
algebra $\SL_q(2)$. Suppose $q$ is generic, that is, it is not a
root of unity or $q = \pm 1$. Then $\rep \SL_q(2)$ is a semisimple
tensor category whose Grothendieck ring is  isomorphic to the
Grothendieck ring of the category $\rep \SL(2)$. The category
$\rep \SL_q(2)$ has a self-dual faithful simple object $V$
corresponding to the standard $2$-dimensional representation of
$\SL(2)$.

Fix an isomorphism $\Phi: X \to X^*$ such that the induced map $\1
\to X \otimes X^* \overset{\Phi \otimes \Phi^{-1}}\to X^* \otimes
X \to \1$ is given by the scalar $-(q + q^{-1})$, where $q \in
k^{\times }$ is \emph{generic}.  In this case there exists a
unique tensor functor $F : \rep \SL_q(2) \to \C$  such that $F(V )
= X$ and $F(\phi) = \Phi$, where $\phi: V \to V^*$ is a fixed
isomorphism in $\rep \SL_q(2)$. See \cite[Theorem 2.1]{EO},
\cite[Chapter XII]{turaev}. Since $X$ is a faithful object and
$F(V) = X$, then the functor $F$ is surjective. Thus every such
fusion category $\C$ is a quotient of $\rep \SL_q(2)$ for an
appropriate value of $q$. (Note that when $q$ is a root of unity,
the category $\rep \SL_q(2)$ has a similar universal property
\cite[Theorem 2.3]{EO}.)
\end{example}

\begin{example}{(Faithful comodules of dimension $2$.)}
Finite dimensional cosemi\-simple Hopf algebras with a self-dual
faithful irreducible comodule $V$ of dimension $2$ were classified
in \cite{BN}.

Let $\nu(V) = \pm 1$ denote the Frobenius-Schur indicator of $V$.
If $\nu(V) = -1$, then  $H$ is commutative and isomorphic to the
dual group algebra $k^{\widetilde \Gamma}$,  where $\widetilde
\Gamma$ is a non-abelian binary polyhedral group.

If, on the other hand, $\nu(V) = 1$, then either $H$ is
commutative and isomorphic to $k^{D_n}$, $n \geq 3$, where $D_n$
is the dihedral group of order $2n$, or $H$ is isomorphic to one
of certain nontrivial Hopf algebra deformations $\mathcal
A[\widetilde \Gamma]$ or $\mathcal B[\widetilde \Gamma]$ of a
binary polyhedral group $\widetilde \Gamma$. In the last case, $H$
fits into an an abelian cocentral exact sequence $k \to k^{\Gamma}
\to H \to k\mathbb Z_2 \to k$, where $\widetilde \Gamma /\mathbb
Z_2 = \Gamma \subseteq \textrm{PSL}_2(k)$ is a finite polyhedral
group of even order. The universal grading group of the category
$\C$ of finite dimensional $H$-comodules is $\mathbb Z_2$, and the
adjoint subcategory is the category of representations of the
commutative Hopf subalgebra $k^{\Gamma}$.
\end{example}

\section{Faithful simple objects and the universal grading
group}\label{s-faithful}

Let $G$ be a finite group and let $\C = \rep G$. Then $U(\C)
\simeq Z(G)$. A classical result says that if $G$ has a faithful
character, then the center $Z(G)$ is cyclic; see for instance
\cite[Theorem 2.32 (a)]{isaacs}. The analogous statement is true
for any fusion category, as follows from the next theorem.

\begin{theorem}\label{G-cyclic} Let $\C$ be a fusion category and let $U(\C)$ be its universal grading group.
Then the following hold:
\begin{enumerate}\item[(i)] Suppose $X \in \Irr (\C)$ and let $g \in U(\C)$ such that $X \in \C_g$.
Then the order of $g$ divides the order of $X$.
\item[(ii)] Suppose $\C$ is generated by simple objects $X_1, \dots, X_m$ as a fusion
category, and let $g_i \in U(\C)$ such that $X_i \in \C_{g_i}$,
$1\leq i \leq m$. Then $g_1, \dots, g_m$ generate the group
$U(\C)$.
\end{enumerate}
In particular, if $X \in \Irr (\C)$ is faithful, then the group
$U(\C)$ is cyclic and its order divides the order of $X$.
\end{theorem}

As a consequence, if $\C$ has a self-dual faithful simple object, then $U(\C) = 1$ or $U(\C) \simeq \mathbb Z_2$.

\begin{proof} (i) Let $n = \ord(X)$, so that $m(\1, X^{\otimes n}) > 0$. On the
other hand, $X^{\otimes n} \in \C_{g^n}$,  and since $\1 \in
\C_e$, we get $\C_{g^n} = \C_e$. Hence $g^n = e$, and  the order
of $g$ divides $n$.

\medbreak (ii) Since $X_i \in \C_{g_i}$, $1\leq i\leq m$, then
$X_{i_1}\otimes \dots \otimes X_{i_t} \in \C_{g_{i_1} \dots
g_{i_t}}$, for all $1 \leq i_1, \dots, i_t \leq m$. Let $g \in
U(\C)$ and let $Y \in \Irr (\C)$ such that $Y \in \C_g$. By
assumption,  $Y$ appears with positive multiplicity in some tensor
product $X_{i_1}\otimes \dots \otimes X_{i_t}$. Then $Y \in
\C_{g_{i_1} \dots g_{i_t}}$ and we get that $g = g_{i_1} \dots
g_{i_t}$. Thus $U(\C) = \langle g_{i_1}, \dots, g_{i_t} \rangle$,
as claimed. \end{proof}

\begin{example} Let $H$ be a semisimple Hopf algebra and let $K \subseteq H$ be the
maximal central Hopf subalgebra of $H$, so that $K \simeq
k^{U(\C)}$, where $\C = \rep H$ is  the fusion category of finite
dimensional representations of $H$. Theorem \ref{G-cyclic} implies
that if $H$ has a faithful irreducible character $\chi$, then $K
\simeq k^{\mathbb Z_m}$, where $m$ divides the order of $\chi$.
(Compare with \cite[Theorem 3.5]{BN}.) \end{example}

\begin{corollary}\label{sd-faithful} Let $p$ be a prime number. Suppose $\C$
 is a nilpotent fusion category such that $\C$  has a  simple
object of order $p$. Then the Frobenius-Perron dimension of $\C$ is divisible by $p$. In particular, if $\C$ has a self-dual simple object, then $\FPdim \C$ is even.
\end{corollary}

Note that the Frobenius-Perron dimension of a nilpotent fusion
category is always an integer.

\begin{proof} Let $X \in \Irr (\C)$ of order $p$.  Since $\C[X]$ is also nilpotent, then
$U(X) \neq 1$ and therefore $U(X) \simeq \mathbb Z_p$, by Theorem
\ref{G-cyclic}. Hence $p$ divides $\FPdim \C[X]$. This implies the
corollary, since $\FPdim \C[X]$ divides $\FPdim \C$
\cite[Proposition 8.15]{ENO}.
\end{proof}

Combining Lemma \ref{udeC} and Theorem \ref{G-cyclic} we get the
following:

\begin{corollary}\label{exp-udec} Let $\C$ be a  fusion category. Then the following hold:
\begin{enumerate}\item[(i)]
Let $n \in \mathbb N$ and suppose  that the order of $X$ divides
$n$, for all $X \in \Irr \C$. Then the exponent of $U(\C)$ divides
$n$. In particular, if all simple objects of $\C$ are self-dual,
then $U(\C)$ is an elementary abelian $2$-group.
\item[(ii)] Let $p$ be a prime number. Suppose that $\C$ is nilpotent and the order of $X$ is a power of $p$,
for all simple object $X$ of $\C$. Then $\FPdim \C = p^m$, for
some $m \geq 1$. \end{enumerate} \end{corollary}

\begin{proof} By Theorem \ref{G-cyclic} each of the groups $U(X)$
is cyclic of order dividing $n$. Thus (i) follows from Lemma
\ref{udeC}. To prove (ii), observe that it follows from Lemma
\ref{udeC} and Theorem \ref{G-cyclic}, that $U(\C)$ is a
$p$-group, hence of order a power of $p$. Since $\C$ is nilpotent,
$\FPdim \C_{\ad} < \FPdim \C$. By induction, we may assume that
$\FPdim \C_{\ad}$ is a power of $p$. Then so is $\FPdim \C$.
\end{proof}

The next lemma gives a sufficient condition for a fusion category
$\C$ to have a self-dual simple object, in terms of the universal
grading of $\C$.

\begin{lemma} Suppose the nilpotency class of $\C$ is $\leq 2$.
Assume in addition that $g \in U(\C)$ is of order $2$ and the rank
of $\C_g$ is $1$. Then, if $X \in \C_g$ is a simple object of
$\C$, $X$ has order $2$. \end{lemma}

\begin{proof} The assumption on the nilpotency class of $\C$ means
that $\C_{\ad} \subseteq \C_{pt}$. Let $\Gamma$ be the set of
isomorphism classes of invertible objects of $\C_{\ad}$. If $X \in
\C_g$, then $\1 \neq X$ and $X^{\otimes 2} \in \C_{g^2} \subseteq
\C_{\ad}$, because $g$ is of order $2$. Since the multiplicity of
an invertible object in $X^{\otimes 2}$ is $\leq 1$, then
$X^{\otimes 2} \simeq \oplus_{s \in S}s$, for some subset $S
\subseteq \Gamma$. In particular $(\FPdim X)^2 = |S|$

On the other hand, since $\C_g$ is of rank $1$, then $(\FPdim X)^2
= \FPdim \C_g = \FPdim \C_{\ad} = |\Gamma|$. Hence $S = \Gamma$,
and thus $m(\1, X^{\otimes 2}) = 1$. This proves the lemma.
\end{proof}

\medbreak Recall that if $G$ is a finite nilpotent group with
cyclic center, then $G$ has a faithful irreducible character
\cite[Theorem 2.32(b) and Problem 4.3]{isaacs}. Our next theorem
establishes the analogous fact for braided fusion categories, thus
giving a partial converse of Theorem \ref{G-cyclic} in this case.

We need first the following lemma.

\begin{lemma}\label{prod-faithful} Let $\C_1, \C_2$ be nilpotent fusion
categories with commutative Grothen\-dieck ring. Suppose that
$X_i$ is a faithful simple object of $\C_i$, $i = 1, 2$. Assume in
addition that the orders of $U(\C_1)$ and $U(\C_2)$ are relatively
prime. Then $X_1 \boxtimes X_2$ is a faithful simple object of
$\C_1 \boxtimes \C_2$.
\end{lemma}

\begin{proof} Let $\C = \C_1 \boxtimes \C_2$ and put $U_i = U(\C_i)$ and $U = U(\C)$. Note that
$\C_{\ad} = (\C_1)_{\ad} \boxtimes (\C_2)_{\ad}$ and $U = U_1
\times U_2$. In particular, $\C$ is also nilpotent.

By Theorem \ref{G-cyclic} the groups $U_1$ and $U_2$ are cyclic.
Since $|U_1|$ and $|U_2|$ are relatively prime, then $U$ is also
cyclic. Moreover, suppose that $X_i \in (\C_i)_{a_i}$, $i = 1, 2$.
Then $\langle a_i \rangle = U_i$ and thus $U = \langle a \rangle$,
where $a = (a_1, a_2)$.

Let $g \in U$, $g = a^m = (a_1^m, a_2^m)$, $m \geq 1$. Then $(X_1
\boxtimes X_2)^{\otimes m} = X_1^{\otimes m} \boxtimes
X_2^{\otimes m}$ is a nonzero object in ${(\C_1)}_{a_1^m}
\boxtimes {(\C_2)}_{a_2^m} = (\C_1 \boxtimes \C_2)_{(a_1^m,
a_2^m)} = \C_g$.

\medbreak Denote by $\D = \C[X_1 \boxtimes X_2]$ the fusion
subcategory of $\C$ generated by $X_1 \boxtimes X_2$. We have
shown that $\D \cap \C_g \neq 0$, for all $g \in U(\C)$. Since
$\C$ is also a nilpotent fusion category with commutative
Grothendieck ring, then $\D = \C$, by Theorem \ref{braided-nilp}.
\end{proof}

\begin{theorem}\label{braided-primepower} Let $\C$ be a
braided nilpotent fusion category such that the group $U(\C)$ is cyclic.
Then $\C$ has a faithful simple object. \end{theorem}

\begin{proof} Suppose first that $\FPdim \C = p^n$, where $p$ is a prime number, $n \geq 0$.
By assumption, $U(\C)$ has a unique subgroup $T$ of index $p$. Let
$X \in \Irr (\C)$ and suppose that $\C[X] \subsetneq \C$. It
follows from Theorem \ref{braided-nilp} that $\phi_X(U(X))
\subsetneq U(\C)$ and therefore  $\phi_X(U(X)) \subseteq T$. By
Lemma \ref{udeC},  $U(\C) = \cup_{X \in \Irr (\C)}\phi_X(U(X))$.
Therefore there must exist some $X \in \Irr (\C)$ with $\C[X] =
\C$. Then the theorem holds in this case.

\medbreak By \cite[Theorem 1.1]{DGNO} a braided nilpotent fusion
category has a unique decomposition into a tensor product of
braided fusion categories whose Frobenius-Perron dimensions are
powers of distinct primes. That is, there exist prime numbers
$p_1, \dots, p_r$, $p_i \neq p_j$ for all $i \neq j$, and an
equivalence of braided fusion categories $\C \simeq \C_{p_1}
\boxtimes \dots \boxtimes \C_{p_r}$, where for all $i = 1, \dots,
r$, $\C_{p_i}$ is a braided fusion category of Frobenius-Perron
dimension $p_i^{n_i}$, for some $n_i \geq 0$.

\medbreak We have an isomorphism $U(\C) \simeq U(\C_{p_1}) \times
\dots \times U(\C_{p_r})$. Therefore the groups $U(\C_{p_i})$ are
also cyclic. As we have already shown, this implies that each
$\C_{p_i}$ has a faithful simple object $X_i$, $i = 1, \dots, r$.
Since the orders of the groups $U(\C_{p_i})$ are relatively prime,
Lemma \ref{prod-faithful} implies that $X_1 \boxtimes \dots
\boxtimes X_r$ is a faithful simple object of $\C$. This finishes
the proof of the theorem.
\end{proof}


When $\C$ is a braided fusion category, the group $U(\C)$ is abelian. We have in this case the following refinement of Corollary \ref{exp-udec}:

\begin{proposition}\label{gen-sd} Let $\C$ be a braided fusion category. Suppose $\C$ is generated by simple objects $X_1, \dots, X_n$.  Then the exponent of $U(\C)$
divides $\textrm{l.c.m.}\{ \ord X_i|\, 1 \leq i \leq n \}$. In particular, if $\C$ is generated by
self-dual simple objects, then $U(\C)$ is an elementary abelian $2$-group. \end{proposition}

\begin{proof} Let $g_i \in U(\C)$ such that $X_i \in \C_{g_i}$.
Then $U(\C)$ is generated by $g_1, \dots, g_n$, and by Theorem
\ref{G-cyclic},  the order of $g_i$ divides $\ord X_i$, for all
$i= 1, \dots, n$. This implies the proposition, since $U(\C)$ is
abelian. \end{proof}

\medbreak When $\C$ is a \emph{modular} category, there is a group
isomorphism $G(\C) \simeq U(\C)$ \cite{gel-nik}. Theorems \ref{G-cyclic} and \ref{braided-primepower} imply the following:

\begin{corollary}\label{gdec} Let $\C$ be a modular category. Then the following hold:
\begin{enumerate}\item[(i)] Suppose $\C$ has a faithful simple object. Then the group $G(\C)$ is cyclic and its order divides the order of $X$. In particular, if $\C$ has a faithful self-dual simple object, then $G(\C) = 1$ or $G(\C) \simeq \mathbb Z_2$.
\item[(ii)] Suppose $\C$ is nilpotent. If $G(\C)$ is cyclic, then $\C$ has a faithful simple object. \qed
\end{enumerate} \end{corollary}

\section{Generalized Tambara-Yamagami fusion rules}\label{generalized-ty}

Let us consider a fusion category $\C$ such that $\C$ is not
pointed and for all non-invertible simple objects $X, Y$ of $\C$,
their tensor product $X \otimes Y$ is a direct sum of invertible
objects. We shall say in this case that $\C$ has \emph{generalized
Tambara-Yamagami fusion rules}. These categories are classified,
up to equivalence of tensor categories, in \cite{liptrap}.
Semisimple Hopf algebras $H$ such that the category $\rep H$ has
generalized Tambara-Yamagami fusion rules were studied in
\cite{ext-ty}.

Let $G$ be the group of invertible objets of $\C$. Then for all $X
\in \Irr (\C)$ we have the relation $X \otimes X^* \simeq
\oplus_{h \in G[X]} h$. In particular, $\C$ is nilpotent of
nilpotency class $2$.

\begin{lemma}\label{type}
\begin{enumerate}\item[(i)] The action of the group $G$ by left (or right) tensor multiplication on the
set $\Irr (\C)- G$ is transitive.
\item[(ii)] There exists a normal subgroup $\Gamma$ of $G$ such that
$G[X] = \Gamma$, for all non-invertible simple object $X$ of $\C$.
In particular $\cd(\C) = \{ 1, \sqrt{|\Gamma|} \}$.
\item[(iii)] $\Irr (\C) = G \cup \{ X_s| \; s
\in G/\Gamma\}$, where $X_{\overline g} = g \otimes X$, $g \in G$,
obeying
\begin{equation}\label{gen-ty}g \otimes h \simeq gh, \quad X_{\overline g}
\otimes X_{\overline h}^* \simeq \oplus_{a \in \Gamma} gah^{-1},
\quad \forall g, h \in G. \end{equation}
\end{enumerate}
\end{lemma}

\begin{proof} Let $X, Y \in \Irr (\C) - G$. By
assumption, $X \otimes Y^* = \oplus_{h \in S}h$ for some subset $S
 \subseteq G$. Then there exists $h \in G$ such that $m(h, X \otimes Y^*) >
 0$. Hence $m(X, h \otimes Y) = m(h, X \otimes Y^*) = 1$ and thus $h \otimes Y = X$.
 This shows that the left action of $G$ is transitive. The statement for the right
 action is proved similarly. This shows (i).

Part (ii) follows from transitivity of the right action, since
$G[X \otimes h] = G[X]$, for all simple object $X$, and for all $h
\in G$. Note that, if $h \in G$, $\Gamma = G[h \otimes X] =
hG[X]h^{-1} = h\Gamma h^{-1}$. Hence $\Gamma$ is normal in $G$.

Finally, let $X$ be a fixed non-invertible simple object and set
$X_{\overline g} = g \otimes X$,  for every $\overline g \in
G/\Gamma$. The isomorphism class of $X_{\overline g}$ is well
defined, since $\Gamma = G[X]$. This also implies that
$X_{\overline g} \simeq X_{\overline h}$ if and only if $\overline
g = \overline h$ in $G/\Gamma$. It is clear that the relations
\eqref{gen-ty} are satisfied. By (i), every non-invertible $Y \in
\Irr (\C)$ is isomorphic to $X_{\overline g}$, for some $g \in G$,
and thus we get (iii).
\end{proof}

\begin{remark}\label{rk-index} Let $G$ and $\Gamma$ be the groups associated to $\C$ as in
Lemma \ref{type}. We shall say that $\C$ has generalized
Tambara-Yamagami fusion rules of \emph{type} $(G, \Gamma)$. In
this case, the lemma implies that the rank of $\C$ is
$[G:\Gamma](1+|\Gamma|)$ and $\FPdim \C = 2|G|$.

In addition, if the index of $\Gamma$ in $G$ is odd, then $\C$ has
a non-invertible simple object of order $2$. \end{remark}

Tambara-Yamagami categories and Moore-Read categories are examples
of fusion categories with generalized Tambara-Yamagami fusion
rules of types $(G, G)$, where $G$ is a finite abelian group, and
$(\mathbb Z_4, \mathbb Z_2)$, respectively.

\begin{proposition}\label{grading} Suppose $\C$ has generalized Tambara-Yamagami fusion
rules of type $(G, \Gamma)$.  Then we have
\begin{enumerate}\item[(i)] The adjoint subcategory $\C_{\ad}$ coincides with $\C[\Gamma]$ and it
is equivalent to the category of $\Gamma$-graded vector spaces.
\item[(ii)] The group $U(\C)$ is of order $2[G:\Gamma]$.
\item[(iii)] The universal grading $\lambda: \Irr(\C) \to U(\C)$
induces an isomorphism $G/\Gamma \simeq \lambda(G)$, such that
$[U(\C): \lambda(G)] = 2$.
\item[(iv)] Let $g \in U(\C)$. Then the rank of $\C_g$ is $2$ if
and only if $g \in \lambda(G)$, and in this case $\C_g = (\tilde
g\Gamma)$, where $\lambda(\tilde g) = g$. Otherwise, $\C_g$ is of
rank $1$.
\end{enumerate}
\end{proposition}

\begin{proof} (i) By Lemma \ref{type} (ii), $X \otimes X^* \simeq \oplus_{h \in \Gamma}
h$, for all non-invertible simple object $X$ of $\C$. Therefore
$\C_{\ad} = \C[\Gamma]$. On the other hand, if $X$ is a
non-invertible simple object, then $X$ gives rise (via left tensor
multiplication) to a fiber functor on $\C[\Gamma]$. Then
$\C[\Gamma]$ is equivalent to the category of $\Gamma$-graded
vector spaces. This shows (i).

(ii) Since $\FPdim \C = |\Gamma| |U(\C)| = 2|G|$, then $|U(\C)| =
2[G:\Gamma]$, as claimed.

(iii) It is clear that $\lambda$ induces a group homomorphism
$\lambda: G \to U(\C)$. Since $\C_{\ad} = \C[\Gamma]$, then $\ker
\lambda = \Gamma$ and  $G/\Gamma \simeq \lambda(G)$. The last
assertion follows from (ii).

(iv) We have $\FPdim \C_g = |\Gamma|$, for all $g \in U(\C)$. This
implies (iv), in view of (iii) and Lemma \ref{type} (ii).
%
%
\end{proof}

%

Recall that a braided fusion category $\mathcal T$ is called \emph{tannakian} if there exists an equivalence of braided tensor categories $\mathcal T \simeq  \rep G$, where $G$ is a finite group.

\begin{lemma}\label{genty-order2} Let $\C$ be a modular category  with generalized Tambara-Yamagami fusion rules. Then we have
\begin{enumerate}\item[(i)] $\Gamma \simeq \mathbb Z_2$ and $\C[\Gamma]$ is not tannakian.

\item[(ii)] $\C$ has a non-invertible simple object of order $2$.
\end{enumerate}
\end{lemma}

\begin{proof} (i) Since $\C$ is  modular, then $G \simeq U(\C)$.
From Proposition \ref{grading}, we get that $|\Gamma| = 2$. Hence $\FPdim X = \sqrt 2$, for all non-invertible simple object $X$ of $\C$.
Suppose on the contrary that $\C[\Gamma]$ is tannakian, that is, $\C[\Gamma] \simeq \rep \mathbb Z_2$ as braided
tensor categories. We may regard $\C$ a fusion subcategory of its Drinfeld center $\Z(\C)$. It follows from \cite[Proposition 2.10]{ENO2}, \cite[Theorem 4.18 (i)]{DGNOI}, that $\C$ is
a $\mathbb Z_2$-equivariantization of a (not necessarily braided) fusion category $\D$. In other words, there exist a fusion category $\D$ and an action of $\rho: \mathbb Z_2 \to \Aut_{\otimes }\D$ by tensor autoequivalences, such that $\C \simeq \D^{\mathbb Z_2}$ as fusion categories.

The forgetful functor $F: \D^G \to \D$  is a dominant tensor
functor and it induces an exact sequence of fusion categories
$\C[\Gamma] \to \D^G \to \D$ \cite{tensor-exact}. Hence $\D$ is
not integral, and in particular, it is not pointed. This implies
that $\D$ also has generalized Tambara-Yamagami fusion rules, say,
of type $(\tilde G, \tilde \Gamma)$.

Let $Y$ be a non-invertible simple object of $\D$. Since $F$ is
dominant,  there exists a (non-invertible) simple object $X$ of
$\C$ such that $m(Y, F(X)) > 0$.

Let $L: \D \to \D^G$ denote the left adjoint of $F$. We have
$FL(Y) = \oplus_{t \in \Gamma}\rho^t(Y)$ \cite{tensor-exact}. It
follows by adjunction, that $X$ is a simple direct summand of
$L(Y)$ and therefore $F(X)$ is a direct summand of $FL(Y)$. Hence
$\FPdim X = n \FPdim Y$, for some natural number $n$. By Lemma
\ref{gen-ty}, $\FPdim Y =  \sqrt{|\tilde \Gamma|}$, and thus we
obtain $\sqrt 2 = n \sqrt{|\tilde \Gamma|}$. Hence $n = 1$. It
follows from this that $F(X) = Y$ is a simple object of $\D$.

On the other hand, we have $X^{\otimes 2} \simeq \oplus_{g \in t\Gamma}g$, for some $t \in G$. Then $F(t)$ is an invertible object of $\D$ and $m(F(t), F(X)^{\otimes 2}) = 2$, because $F(g) \simeq \1$, for all $g \in \Gamma$. Thus we reach again a contradiction, because an invertible object can only appear with multiplicity $0$ or $1$ in $Y^{\otimes 2}$. This contradiction shows that $\C$ cannot be a $\mathbb  Z_2$-equivariantization of a fusion category $\D$. Then $\C[\Gamma]$ is not tannakian.

\medbreak (ii)  Let $\C_{pt}'$ denote the M\" uger centralizer of
$\C_{pt} = \C[G]$ in $\C$. By \cite[Corollary 3.27]{DGNOI},
\cite[Corollary 6.8]{gel-nik}, we have $\C[G]' = \C_{\ad} =
\C[\Gamma]$. In particular, $\C[\Gamma] \subseteq \C[G]$ coincides
with the M\" uger center of $\C[G]$. Since, by Part (i),
$\C[\Gamma]$ is not tannakian, then $\C[G]$ is \emph{slightly
degenerate}, and it follows from \cite[Proposition 2.6 (ii)]{ENO2}
that
 $\C[G] \simeq \C[\Gamma]
\boxtimes \B$ as braided tensor categories, where $\B$ is a
pointed modular category. Note that if $\FPdim \B = 1$, then $G =
\Gamma$ and $\FPdim \C = 4$. Therefore $\C$ is an Ising category
and (ii) holds in this case. Hence we may assume that $\FPdim \B =
1$.

On the other hand, the fact that $\B$ is modular implies that $\C
\simeq \B \boxtimes \B'$ as braided tensor categories, where $\B'$
is the M\" uger centralizer of $\B$ in $\C$ \cite[Proposition
4.1]{mueger-structure}. Then $\B'$ is modular and clearly it has
generalized Tambara-Yamagami fusion rules.

Since $\FPdim \B' < \FPdim \C$, we may inductively assume that $\B'$ has a non-invertible simple object of order $2$. Then so does $\C$. This proves (ii) and finishes the proof of the lemma. \end{proof}

The following theorem is the main result of this section. Combined
with the results in \cite[Subsection 2.11 and Appendix B]{DGNOI},
\cite{quinn}, it gives the classification of modular
categories with generalized Tambara-Yamagami fusion rules.

\begin{theorem}\label{modular-genty} Let $\C$ be a modular
category. Then $\C$ has generalized Tambara-Yamagami fusion rules if and only if  $\C \simeq \mathcal I \boxtimes \mathcal B$, where $\mathcal
I$ is an Ising category and $\mathcal B$ is a pointed modular category.
\end{theorem}

\begin{proof} Suppose that $\C = \mathcal I \boxtimes \mathcal B$, where
$\mathcal I$ is an Ising category and $\mathcal B$ is a pointed
category.  Then every simple object of $\C$ is isomorphic to $Y
\boxtimes g$, where $Y \in \Irr(\mathcal I)$ and $g \in G(\mathcal
B)$. Then $\C$ is not pointed, and the non-invertible simple
objects of $\C$ are represented by $X \boxtimes g$, where $X$ is
the unique non-invertible simple object of $\mathcal I$ and $g \in
G(\mathcal B)$. This implies that $\C$ has generalized
Tambara-Yamagami fusion rules of type $(G(\mathcal B), \mathbb
Z_2)$. If in addition $\mathcal B$ is modular category, then so is
$\C$. This proves the 'if' direction.

\medbreak Conversely, suppose that  $\C$ is  modular and has generalized Tambara-Yamagami fusion rules. By Lemma \ref{genty-order2} we have $|\Gamma| = 2$ and there exists
a non-invertible simple object $X \in \C$ of order $2$. Then the
fusion subcategory $\mathcal I = \C[X]$ is an Ising category, and
it is necessarily modular, by \cite[Corollary B.12]{DGNOI}.

By \cite[Theorem 4.2]{mueger-structure}, there is an equivalence
of ribbon categories $\C \simeq \mathcal I \boxtimes \mathcal B$,
where $\mathcal B = \mathcal I'$ is the M\" uger centralizer of
$\mathcal I$ in $\C$. Furthermore, since $\C$ is modular, then so
is $\mathcal I' = \mathcal B$. Note that $\cd(\mathcal I) = \{ 1,
\sqrt{2}\} = \cd(\C)$. Therefore $\mathcal B$ must be pointed.
This finishes the proof of the theorem. \end{proof}


\section{Modularization of braided Tambara-Yamagami categories}\label{ty-mod}

Along this section, $\C = \mathcal{TY}(G, \tau, \chi)$ will be a
Tambara-Yamagami fusion category, where $G$ is a finite abelian
group, $\tau$ is a square root of the order of $G$ in $k$ and
$\chi: G \times G \to k^{\times}$ is a non-degenerate symmetric
bicharacter on $G$ \cite{TY}.

We assume that $\C$ is braided. All possible structures of braided
category in $\C$ are classified in \cite{Siehler-braided}. In
particular, $G$ is an elementary abelian $2$-group, and there are
two choices of compatible ribbon structures. Let us consider a
fixed choice $\theta \in \Aut(\id_\C)$, so that $\C$ becomes a
premodular category.


\medbreak Let $\C' \subseteq \C$ be M\" uger center of $\C$. In
the terminology of \cite{bruguieres},   $\C'$ is the fusion
subcategory of transparent objects of $\C$.

\begin{lemma}\label{transp-ty} Suppose $\C$ is not symmetric. Then we have $\C' = \C[T]$, where $T$ is
the subgroup of $G$ defined by $T = \{ g \in G|\, \chi(g, g) =
1\}$. Moreover, the category $\C$ is modularizable.
\end{lemma}

\begin{proof} Since $\C$ is not symmetric, and $X$ generates $\C$, then $X \notin \C'$. Hence
$\C' \subseteq \C[G]$.

Observe that an object $Z$ belongs to $\C'$ if and only if $Z$
centralizes $X$. It follows from  \cite[Subsection
3.1]{Siehler-braided} that, after a suitable normalization, the
bradings $\sigma_{g, X}: g \otimes X \to X \otimes g$ and
$\sigma_{X, g}: X \otimes g \to g \otimes X$ correspond, under the
identification $g \otimes X = X = X \otimes g$, to $s(g)\id_X$,
where $s(g) \in k^{\times}$ are such that $s(g)^2 = \chi(g, g)$,
for all $g \in G$. This implies that $\C' = \C[T]$, where $T = \{
g \in G|\, \chi(g, g) = 1\}$, as claimed.

Let $\theta \in \Aut(\id_\C)$ be the ribbon structure of $\C$.
Then $\theta_g = s(g)^2$ \cite[Subsection 3.5]{Siehler-braided},
for all $g \in G$. Hence $\theta_g = 1$ for all $g \in T$ (this
can also be deduced from \cite[Lemma 5.4]{mueger}, since all
simple objects $g$ of $\C'$ are invertible and satisfy $g \otimes
X \simeq X$). This implies that $\C'$ is tannakian and thus $\C$
is modularizable \cite[Th\' eor\` eme 3.1]{bruguieres}.
\end{proof}

\begin{remark} Lemma \ref{transp-ty} implies that if $\C$ is modular
then $\chi(g, g) \neq 1$, for all $1\neq g \in G$. On the other
hand, if $\chi(g, g) = 1$, for all $g \in G$, then either $\C$ is
symmetric or $\tilde \C$ is pointed and $\FPdim \tilde \C = 2$.
\end{remark}

Let $F: \C \to \tilde \C$ denote the modularization functor. So
that $\tilde \C$ is a modular category and $F$ is a dominant
braided tensor functor. We have in addition

\begin{proposition}\label{equiv}  There is an action
$\rho: \underline{T} \to \underline{\Aut}_{\otimes}\tilde \C$ by
braided autoequivalences such that $\C \simeq \tilde \C^T$ as
braided tensor categories over $\rep T$.
\end{proposition}

\begin{proof} It follows from the results in \cite{bruguieres} that there is an
exact sequence of braided tensor functors $\rep T \to \C \overset{F}\to \tilde \C$.
Then the proposition follows from \cite[Corollary 5.31]{tensor-exact}.
See \cite[Example 5.33]{tensor-exact}.  \end{proof}

\begin{lemma}\label{not-pointed} Suppose $\tilde \C$ is not pointed. Then
$G(\tilde \C) \simeq G/T \simeq \mathbb Z_2$.
\end{lemma}

\begin{proof} By Proposition \ref{equiv}, $\C \simeq \tilde \C^T$
is a $T$-equivariantization. The modularization functor
corresponds to the forgetful functor $F: \tilde \C^T \to \tilde
\C$. Since, by assumption, $\tilde \C$ is not pointed, then $T
\neq G$.

Note that $\C' = \C[T]$ is the kernel of  $F$ in the sense of
\cite{tensor-exact}. Then $F(G) \simeq G/T$ is isomorphic to a
subgroup $G(\tilde \C)$.

Let $h \in G(\tilde \C)$ be an invertible object. We claim that
$m(h, F(X)) = 0$. This can be seen as follows. Let $L: \C \to
\C^G$ denote the left adjoint of $F$. Then we have $FL(h) =
\oplus_{t \in T}\rho^t(h)$ \cite{tensor-exact}, and in particular,
$FL(h)$ belongs to $\tilde \C_{pt}$. Suppose on the contrary that
$m(h, F(X)) > 0$. It follows by adjunction, that $X$ is a simple
direct summand of $L(h)$ and therefore $F(X)$ is a direct summand
of $FL(h)$. This implies that $F(X) \in \tilde \C_{pt}$ and then,
by surjectivity of $F$,  $\tilde \C_{pt} = \tilde \C$, since $X$
generates $\C$. This contradiction shows that $m(h, F(X)) = 0$, as
claimed.

By surjectivity of the functor $F$, there exists $g \in G$ such
that $m(h, F(g)) > 0$. Then $h \simeq F(g) \in G/T$. This shows
that $G(\tilde \C) \simeq G/T$.

Observe next that, since $G$ is an elementary abelian $2$-group,
then $\chi$ induces a group homomorphism $f: G \to \mathbb Z_2$,
defined in the form $f(g) = \chi(g, g)$, for all $g\in G$. We have
$T = \ker f$, whence $[G: T] = 2$, because $T \neq G$.
\end{proof}

\begin{theorem}\label{modulariz-ty} Let $\C$ be a braided Tambara-Yamagami fusion
category and let $\tilde \C$ be the modularization of $\C$. Then
we have:
\begin{enumerate}\item[(i)] $\C$ is integral if and only if $\tilde
\C$ is pointed.
\item[(ii)] Suppose that $\C$ is not integral. Then $\tilde \C \simeq \mathcal
I$, as braided tensor categories, where $\mathcal I$ is an Ising
category.
\end{enumerate}
\end{theorem}

Note in addition that the integrality of $\C$ is determined by the
parity of the rank of $G$, namely,  $\C$ is integral if and only
if the rank of $G$ is even.

\begin{proof} (i) We have an exact sequence of fusion categories
$\C[T] \to \C \to \tilde \C$. By \cite{tensor-exact} $\C$ is
integral if and only if $\tilde \C$ is integral. If $\tilde \C$ is
not pointed, then by Theorem  \ref{modular-genty}, it contains an
Ising subcategory and therefore it is not integral. Hence $\tilde
\C$ is integral if and only if it is pointed. This shows (i).

\medbreak (ii) Since $\C$ is not integral, then $\tilde \C$ is not
integral neither. In particular, $\tilde \C$ is not pointed. By
Lemma \ref{not-pointed} we have $G(\tilde \C) \simeq G/T \simeq
\mathbb Z_2$. The fusion rules of $\C$ imply that $\tilde \C$ has
generalized Tambara-Yamagami fusion rules.  Theorem \ref{modular-genty} implies that  $\tilde \C$ is equivalent to an Ising category.
This proves (ii) and finishes the proof of the theorem.
\end{proof}

\bibliographystyle{amsalpha}

\end{document}